\documentclass[12pt]{amsart}
\usepackage{latexsym, amssymb, amsmath}

\newtheorem{theorem}{Theorem}[section]
\newtheorem{lemma}[theorem]{Lemma}

\newtheorem{remark}[theorem]{Remark}
\theoremstyle{definition}

\newtheorem*{example}{Example}

\makeatletter
    
    \@addtoreset{equation}{section}
  \makeatother

\begin{document}

\title[Biharmonic Wintgen ideal submanifolds]
{Biharmonic Wintgen ideal submanifolds in Riemannian manifolds of constant sectional curvature}

\author{Shun Maeta}
\address{Department of Mathematics, Chiba University, 1-33, Yayoicho, Inage, Chiba, 263-8522, Japan.}
\email{shun.maeta@faculty.gs.chiba-u.jp~{\em or}~shun.maeta@gmail.com}

\date{}

\dedicatory{}


\begin{abstract}
In this paper, we show that every biharmonic Wintgen ideal submanifold of dimension at least three or dimension two with codimension at most two in a Riemannian manifold of nonpositive constant sectional curvature is minimal.
Under the same dimensional assumption, we also prove that every biharmonic Wintgen ideal submanifold in a Riemannian manifold of positive constant sectional curvature has constant mean curvature on each connected component.
This gives partial affirmative answers to Chen's conjecture, to the generalized Chen's conjecture in hyperbolic spaces, and to the Balmu\c{s}-Montaldo-Oniciuc conjecture in spheres within the class of Wintgen ideal submanifolds.
\end{abstract}

\maketitle

\section{Introduction}
A smooth map $\phi:(M,g)\rightarrow (N,h)$ is called biharmonic if it is a critical point of the bienergy functional
\[
E_2(\phi)=\frac{1}{2}\int_M |\tau(\phi)|^2\,dv_g,
\]
where $\tau(\phi)$ is the tension field of $\phi$ (cf. \cite{ES1964}, \cite{EL1983}).
If an isometric immersion $\phi:(M,g)\rightarrow (N,h)$ is biharmonic, then $M$ is called a biharmonic submanifold of $N$.
Every minimal submanifold is biharmonic.
The basic problem is to determine when the converse holds, or when biharmonicity forces constant mean curvature in Riemannian manifolds of positive curvature.

A classical inequality in the theory of submanifolds is
\[
K\leq |{\bf H}|^2
\]
for surfaces in Euclidean three-space $\mathbb{E}^3$, where $K$ is the Gaussian curvature and ${\bf H}$ is the mean curvature vector.
Equality holds exactly at umbilical points.

In higher codimension, the Wintgen inequality refines this classical inequality by incorporating an additional term involving the normal curvature.
It was proved by Wintgen \cite{Wintgen1979} for surfaces in four-dimensional Euclidean space and by Guadalupe and Rodriguez \cite{GR1983} for surfaces in Riemannian manifolds of constant sectional curvature.
For related background on Wintgen-type inequalities, see Chen \cite{Chen2021}.

Surfaces attaining equality in the Wintgen inequality are called Wintgen ideal surfaces.
The equality case is geometrically meaningful rather than merely algebraic.
In codimension two, it is equivalent to the circularity of the curvature ellipse.
Surfaces with a circular curvature ellipse are often called superconformal surfaces.
Thus Wintgen ideal surface theory is closely related to conformal surface theory, superminimal surfaces, and minimal surface theory.
Dajczer and Tojeiro described superconformal surfaces in $\mathbb{R}^4$ in terms of conjugate minimal surfaces \cite{DT2009}.
In the spherical case, Bryant studied superminimal surfaces in $\mathbb{S}^4$ \cite{Bryant1982}.
Although superminimality is usually defined in twistor-theoretic terms, it is equivalent to a condition on the second fundamental form \cite{MU1997}.
In codimension two, this condition says precisely that the surface is minimal and its curvature ellipse is a circle.
See, for example, \cite{Chen2010} and \cite{LMWX2016}.
Thus, in the terminology of Wintgen geometry, superminimal surfaces in $\mathbb{S}^4$ are minimal Wintgen ideal surfaces.

This relation provides one motivation for the present paper.
Bryant's work \cite{Bryant1982} provides a classical model for the minimal case of the geometry considered here.
The purpose of the present paper is to study the biharmonic analogue of this picture and to treat the corresponding rigidity problem for Wintgen ideal submanifolds.
For submanifolds of dimension at least three, the Wintgen ideal condition is understood as equality in the DDVV inequality.
This equality case gives the Choi--Lu normal form for the shape operators (cf. \cite{CL2008}, \cite{DPSV2012}, \cite{GT2008}).
This normal form reflects the special geometry of Wintgen ideal submanifolds and makes them a natural higher-codimensional class in which to study biharmonic rigidity.
The Wintgen ideal condition is not merely an algebraic equality case of the DDVV inequality, but defines a rich and highly structured class of higher-codimensional submanifolds. An important feature is that the Wintgen ideal condition is invariant under conformal transformations, which places its study naturally in the framework of M\"obius geometry. 
This viewpoint has led to substantial construction and classification results. 
In particular, Xie, Li, Ma, and Wang \cite{XLMW2021} constructed new compact examples from real representations of ${\rm SU}(2)$, introduced a frame sequence that reveals a connection with the classical harmonic sequence of Riemann surfaces, and obtained a complete classification of M\"obius homogeneous Wintgen ideal submanifolds of dimension at least three. 
Together with the close relations to superconformal and superminimal geometry described above, these developments show that Wintgen ideal submanifolds constitute a distinguished and substantial class in higher-codimensional submanifold geometry.

The second motivation comes from Chen-type conjectures for biharmonic submanifolds.
Chen's conjecture predicts that every biharmonic submanifold in a Euclidean space is minimal.
The generalized Chen's conjecture asserts that every biharmonic submanifold in a Riemannian manifold of nonpositive sectional curvature is minimal.
This conjecture is false in general, as shown by examples due to Ou and Tang \cite{OT2012}, but the corresponding problem for Riemannian manifolds of constant nonpositive sectional curvature, particularly hyperbolic spaces, remains open in full generality.
In positive curvature, the Balmu\c{s}-Montaldo-Oniciuc conjecture predicts that proper biharmonic submanifolds in spheres have constant mean curvature (cf. \cite{BMO2008}, \cite{Chen2014}).
Thus our problem can also be viewed as a Chen-type rigidity problem under the additional Wintgen ideal condition.

A large part of the known progress on these conjectures concerns hypersurfaces.
For Chen's conjecture in Euclidean spaces, we refer to \cite{CI1991}, \cite{Jiang1987}, \cite{HV1995}, \cite{Defever1998}, \cite{Dimitric1992}, \cite{DA2020}, \cite{FHZ2021}, and \cite{FHZ2023}.
For the nonpositive-curvature setting and the hyperbolic case, see \cite{Oniciuc2002}, \cite{Luo2014}, and \cite{FHZ2023}.
For the Balmu\c{s}-Montaldo-Oniciuc conjecture in spheres, see \cite{CMO2001}, \cite{BMO2008}, \cite{BMO2010}, and \cite{FHZ2023}.
These references indicate that the hypersurface case has been studied extensively.

On the other hand, the conjectures are not only hypersurface problems.
A full understanding requires results in higher codimension.
Several partial affirmative answers are known in higher codimension.
Akutagawa and the author proved that every biharmonic properly immersed submanifold in a Euclidean space is minimal \cite{AM2013}.
In nonpositive curvature, Nakauchi and Urakawa proved minimality for complete biharmonic submanifolds under a square-integrability condition on the mean curvature vector \cite{NU2013}, and Luo also obtained related results under $L^p$ and volume-growth assumptions \cite{Luo2015}.
The author proved that every properly immersed biharmonic submanifold in a hyperbolic space is minimal \cite{Maeta2014}.
Under the assumption that the normalized mean curvature vector is parallel, Chen's conjecture has also been studied for higher-codimensional submanifolds \cite{MU2013}, \cite{ST2018}, \cite{Chen2019}.
Balmu\c{s}, Montaldo, and Oniciuc classified biharmonic submanifolds in spheres with parallel normalized mean curvature vector fields and with at most two distinct principal curvatures, and provided an affirmative answer to the Balmu\c{s}-Montaldo-Oniciuc conjecture for biharmonic submanifolds in spheres \cite{BMO2013}.
These results show that higher-codimensional submanifolds form a natural but still less developed setting for Chen-type problems.
In this paper, we provide partial affirmative answers to Chen's conjecture, to the corresponding problem in hyperbolic spaces, and to the Balmu\c{s}--Montaldo--Oniciuc conjecture in spheres within the class of Wintgen ideal submanifolds.

\begin{theorem}\label{main}
Let $M^m$ be a Wintgen ideal submanifold in an $(m+p)$-dimensional Riemannian manifold $N^{m+p}(\rho)$ of constant sectional curvature $\rho$, where $m\geq3$ and $p\geq1$, or $m=2$ and $p=1,2$.
If $M^m$ is biharmonic, then the following assertions hold.
\begin{itemize}
\item[$(1)$] If $\rho\leq 0$, then $M^m$ is minimal.
\item[$(2)$] If $\rho>0$, then $|{\bf H}|$ is constant on each connected component of $M^m$.
\end{itemize}
\end{theorem}
\begin{remark} 
In the positive-curvature case, Theorem \ref{main} can be combined with a classification theorem of Xie \cite{Xie2017}. 
More precisely, when $m\geq3$ and $p\geq2$, Theorem 4.6 of \cite{Xie2017} shows that a Wintgen ideal submanifold with constant mean curvature is locally either totally umbilical, minimal in the ambient space form, or contained in an umbilical hypersurface as a minimal Wintgen ideal submanifold. 
Consequently, if $\rho>0$, $m\geq3$, and $p\geq2$, Theorem \ref{main} places every biharmonic Wintgen ideal submanifold in this classification.
\end{remark}

Here and throughout this paper, $N^n(\rho)$ denotes an $n$-dimensional Riemannian manifold of constant sectional curvature $\rho$.
No completeness or connectedness of the ambient manifold is assumed.
Throughout this paper, all manifolds and maps are assumed to be smooth.

\section{Preliminaries}
Let $M^m$ be an immersed submanifold in $N^{m+p}(\rho)$.
We write $B$, $A_\xi$, ${\bf H}$, $\nabla$, and $\nabla^\perp$ for the second fundamental form, the shape operator with respect to a normal vector $\xi$, the mean curvature vector, the Levi-Civita connection of $M^m$, and the normal connection, respectively.

For a normal vector field $\eta$, we use the convention
\[
\Delta^\perp\eta=
\sum_{i=1}^m
\left\{
\nabla^\perp_{e_i}\nabla^\perp_{e_i}\eta-
\nabla^\perp_{\nabla_{e_i}e_i}\eta
\right\},
\]
where $\{e_1,\cdots,e_m\}$ is a local orthonormal tangent frame.
By the biharmonic submanifold equation of Jiang \cite{Jiang1986} (see also \cite{Oniciuc2002, MO2006}), an $m$-dimensional submanifold of $N^{m+p}(\rho)$ is biharmonic if and only if
\begin{equation}\label{normal}
\Delta^\perp {\bf H}
-\sum_{i=1}^mB(A_{\bf H}e_i,e_i)
+m\rho {\bf H}=0
\end{equation}
and
\begin{equation}\label{tangent}
2\sum_{i=1}^mA_{\nabla^\perp_{e_i}{\bf H}}e_i
+\frac{m}{2}\nabla |{\bf H}|^2=0.
\end{equation}

We recall the higher-dimensional Wintgen ideal condition.
Let $R$ be the normalized scalar curvature of $M^m$, that is
\[
R=
\frac{2}{m(m-1)}
\sum_{1\leq i<j\leq m}
\langle R^M(e_i,e_j)e_j,e_i\rangle,
\]
where $R^M$ is the curvature tensor of $M^m$.
Let $\nu^\perp$ be the normalized normal scalar curvature.
For a local orthonormal normal frame $\{\xi_1,\cdots,\xi_p\}$, it is defined by
\[
\nu^\perp=
\frac{2}{m(m-1)}
\left(
\sum_{1\leq i<j\leq m}
\sum_{1\leq \alpha<\beta\leq p}
\left\langle
R^\perp(e_i,e_j)\xi_\alpha,\xi_\beta
\right\rangle^2
\right)^{1/2},
\]
where $R^\perp$ is the curvature tensor of the normal connection.
These quantities are independent of the chosen orthonormal frames.
The DDVV inequality in a Riemannian manifold of constant sectional curvature is
\begin{equation}\label{ddvv}
R+\nu^\perp\leq |{\bf H}|^2+\rho.
\end{equation}
A submanifold is called Wintgen ideal if equality holds identically in \eqref{ddvv}.
The equality case of the DDVV inequality is described by the Choi--Lu normal form (cf. \cite{CL2008}, \cite{DPSV2012}, \cite{GT2008}).
If $p=1$, equality in \eqref{ddvv} is equivalent to total umbilicity.

If $p\geq2$, then at each point there are orthonormal bases of the tangent and normal spaces such that
\begin{equation}\label{CL1}
A_{\xi_1}=\lambda_1I_m+\mu S_1,
\qquad
A_{\xi_2}=\lambda_2I_m+\mu S_2,
\qquad
A_{\xi_\alpha}=\lambda_\alpha I_m
\quad (\alpha\geq3),
\end{equation}
where
\[
S_1e_1=e_2,
\qquad
S_1e_2=e_1,
\qquad
S_1e_u=0
\quad (u\geq3),
\]
and
\[
S_2e_1=e_1,
\qquad
S_2e_2=-e_2,
\qquad
S_2e_u=0
\quad (u\geq3).
\]
The Choi--Lu normal form is a pointwise statement.
In the differential arguments below, whenever covariant derivatives of the coefficients in \eqref{CL1} are used, we work locally as follows.
At any point of the relevant set at which $\mu\neq0$, we restrict to a sufficiently small connected open neighbourhood on which $\mu$ does not vanish.
On this neighbourhood, we choose smooth local tangent and normal orthonormal frames for which \eqref{CL1} holds throughout.
Such a smooth adapted frame is called a Choi--Lu frame.
All differential computations involving the coefficients $\lambda_\alpha$ and $\mu$ are understood in such a local Choi--Lu frame.
For more details, see \cite{CL2008}, \cite{DPSV2012}, and \cite{GT2008}.
When $m=2$, the corresponding adapted frame is the usual smooth frame for Wintgen ideal surfaces on a non-umbilical open set.
It is obtained from the circularity of the curvature ellipse after restricting the neighbourhood further if necessary.
For this surface normal form, see also \cite{GR1983}, \cite{DT2009}, and \cite{Chen2010}.
Consequently, every contradiction or local constancy statement obtained below is valid locally near each point under consideration.
We put
\begin{equation}\label{HH0}
{\bf H}=\sum_{\alpha=1}^p\lambda_\alpha\xi_\alpha,
\qquad
{\bf H}_0=\lambda_1\xi_1+\lambda_2\xi_2.
\end{equation}
A sum over $\alpha=3,\cdots,p$ is understood to be zero when $p=2$.
The corresponding expression for the second fundamental form is
\begin{equation}\label{CLB}
\begin{aligned}
B(e_1,e_1)&={\bf H}+\mu\xi_2,\qquad
B(e_1,e_2)=\mu\xi_1,\qquad
B(e_2,e_2)={\bf H}-\mu\xi_2,\\
B(e_u,e_u)&={\bf H}\quad (u\geq3),\qquad
B(e_i,e_j)=0\quad (i\neq j,\ \{i,j\}\neq\{1,2\}).
\end{aligned}
\end{equation}
Here the last formula is understood together with the symmetry of $B$.

We also recall Chen's allied mean curvature vector.
Let $M^m$ be a submanifold of codimension $p\geq2$, and let $\zeta$ be a normal vector field.
At a point where $\zeta\neq0$, choose a local orthonormal normal frame
\[
\left\{
\eta_1=\frac{\zeta}{|\zeta|},\eta_2,\cdots,\eta_p
\right\}.
\]
The allied vector field of $\zeta$ is defined by
\[
a(\zeta)
=
\frac{|\zeta|}{m}\sum_{\beta=2}^{p}
\operatorname{trace}(A_{\eta_1}A_{\eta_\beta})\eta_\beta.
\]
At points where $\zeta=0$, we set $a(\zeta)=0$.
A submanifold is called a Chen submanifold if the allied mean curvature vector $a({\bf H})$ vanishes identically.

The following algebraic consequences will be used repeatedly.

\begin{lemma}\label{standard}
Assume that the Choi--Lu normal form \eqref{CL1} holds.
Then
\begin{equation}\label{BAH}
\sum_{i=1}^mB(A_{\bf H}e_i,e_i)
=
m|{\bf H}|^2{\bf H}+2\mu^2{\bf H}_0.
\end{equation}
Consequently, the normal biharmonic equation becomes
\begin{equation}\label{normal2}
\Delta^\perp {\bf H}
-
m|{\bf H}|^2{\bf H}
-
2\mu^2{\bf H}_0
+
m\rho {\bf H}=0.
\end{equation}
\end{lemma}

\begin{proof}
For each normal direction $\xi_\beta$, we have
\[
\left\langle
\sum_{i=1}^mB(A_{\bf H}e_i,e_i),\xi_\beta
\right\rangle
=
\operatorname{trace}(A_{\bf H}A_{\xi_\beta}).
\]
Using
\[
A_{\bf H}=|{\bf H}|^2I_m+\mu(\lambda_1S_1+\lambda_2S_2),
\]
together with
\[
\operatorname{trace}S_1=\operatorname{trace}S_2=0,
\quad
\operatorname{trace}(S_1S_2)=0,
\quad
\operatorname{trace}(S_1^2)=\operatorname{trace}(S_2^2)=2,
\]
we obtain
\[
\operatorname{trace}(A_{\bf H}A_{\xi_1})
=
m\lambda_1|{\bf H}|^2+2\mu^2\lambda_1,
\]
\[
\operatorname{trace}(A_{\bf H}A_{\xi_2})
=
m\lambda_2|{\bf H}|^2+2\mu^2\lambda_2,
\]
and
\[
\operatorname{trace}(A_{\bf H}A_{\xi_\alpha})
=
m\lambda_\alpha|{\bf H}|^2
\quad (\alpha\geq3).
\]
This proves \eqref{BAH}.
Substituting \eqref{BAH} into \eqref{normal} gives \eqref{normal2}.
\end{proof}

\begin{lemma}\label{bihwintchen}
Assume that $M^m$ is a biharmonic Wintgen ideal submanifold in a Riemannian manifold of constant sectional curvature, where $m\geq3$ and $p\geq2$, or $m=p=2$. 
Let $W$ be an open set on which  a smooth Choi--Lu frame is chosen and $\mu(x)\neq0.$
Then, at every point $x\in W$, either
\[
{\bf H}_0(x)=0
\]
or
\[
{\bf H}(x)={\bf H}_0(x).
\]
Moreover, on the set 
\[
\tilde W=\{x\in W~\mid~{\bf H}(x)-{\bf H}_0(x)\neq0\},
\]
we have
\[
|{\bf H}|^2=\rho>0,\qquad
{\bf H}_0=0,\qquad
\nabla^\perp{\bf H}=0.
\]
\end{lemma}

\begin{proof}
If $p=2$, then ${\bf H}={\bf H}_0$ and $\tilde W=\emptyset$, hence the assertion is immediately obtained.

Therefore we will consider $m,p\geq3$.

Set 
\[
\mathcal{A}_i^\beta=\langle \nabla^\perp_{e_i}{\bf H}, \xi_\beta \rangle, \qquad
\mathcal{B}_i^\beta=\langle \nabla^\perp_{e_i}\xi_1, \xi_\beta \rangle,\qquad 
\mathcal{C}_i^\beta=\langle \nabla^\perp_{e_i}\xi_2, \xi_\beta \rangle. 
\]
By \eqref{CLB}, 
\[
B(e_j,e_k)
=\delta_{jk}{\bf H}+\mu\langle S_1e_j,e_k\rangle \xi_1+\mu\langle S_2e_j,e_k\rangle \xi_2.
\]
For $\alpha\geq3$, since 
\[
\langle (\nabla_{e_i}B)(e_j,e_k),\xi_\alpha\rangle
=
\delta_{jk}\mathcal{A}_i^\alpha+\mu\langle S_1e_j,e_k\rangle \mathcal{B}_i^\alpha+\mu\langle S_2e_j,e_k\rangle \mathcal{C}_i^\alpha,
\]
by the Codazzi equation, we have
\begin{align*}
0&=
\delta_{jk}\mathcal{A}_i^\alpha+\mu\langle S_1e_j,e_k\rangle \mathcal{B}_i^\alpha+\mu\langle S_2e_j,e_k\rangle \mathcal{C}_i^\alpha\\
&\quad-\Big(
\delta_{ik}\mathcal{A}_j^\alpha+\mu\langle S_1e_i,e_k\rangle \mathcal{B}_j^\alpha+\mu\langle S_2e_i,e_k\rangle \mathcal{C}_j^\alpha
\Big).
\end{align*}
For $u\geq3$, 
we obtain the following identities by choosing $(i,j,k)$ as indicated.
\begin{enumerate}
\item
When $(i,j,k)=(1,u,u)$, we have $\mathcal{A}_1^\alpha=0$.
\item
When $(i,j,k)=(2,u,u)$, we have $\mathcal{A}_2^\alpha=0$.
\item
When $(i,j,k)=(u,1,1)$, we have $\mathcal{A}_u^\alpha+\mu\mathcal{C}_u^\alpha=0$.
\item
When $(i,j,k)=(u,2,2)$, we have $\mathcal{A}_u^\alpha-\mu\mathcal{C}_u^\alpha=0$.
\end{enumerate}
Combining the third with the fourth yields 
\[
\mathcal{C}_u^\alpha=0,\qquad
\mathcal{A}_u^\alpha=0.
\]
We also obtain
\begin{enumerate}
\item
When $(i,j,k)=(u,1,2)$, $\mathcal{B}_u^\alpha=0$.
\item
When $(i,j,k)=(1,2,1)$, $\mathcal{C}_2^\alpha=\mathcal{B}_1^\alpha$.
\item
When $(i,j,k)=(1,2,2)$, $\mathcal{C}_1^\alpha=-\mathcal{B}_2^\alpha$.
\end{enumerate}
Hence,
\[
\nabla^\perp_{e_i}{\bf H}\in{\rm span}\{\xi_1,\xi_2\}.
\]
Set
\[
\mathcal{D}_i=\langle \nabla^\perp_{e_i}\xi_1,\xi_2\rangle,
\qquad
\omega_{ij}^k=\langle \nabla_{e_i}e_j,e_k\rangle.
\]
By definition, $\omega_{ij}^j=0$ and $\omega_{ij}^k=-\omega_{ik}^j$.
The $\xi_1$- and $\xi_2$-components of the Codazzi equation 
\[
(\nabla_{e_1}B)(e_2,e_2)=(\nabla_{e_2}B)(e_1,e_2)
\] 
are 
\begin{align}
&\mathcal{A}_1^1+\mu\mathcal{D}_1+2\omega_{11}^2\mu-e_2\mu=0,\label{cad1}\\
&\mathcal{A}_1^2-\mu\mathcal{D}_2-2\omega_{21}^2\mu-e_1\mu=0.\label{cad2}
\end{align}
The $\xi_1$- and $\xi_2$-components of the Codazzi equation 
\[
(\nabla_{e_1}B)(e_2,e_1)=(\nabla_{e_2}B)(e_1,e_1)
\] 
are 
\begin{align}
&-\mathcal{A}_2^1+\mu\mathcal{D}_2+2\omega_{21}^2\mu+e_1\mu=0,\label{cad3}\\
&-\mathcal{A}_2^2+\mu\mathcal{D}_1+2\omega_{11}^2\mu-e_2\mu=0.\label{cad4}
\end{align}
Combining $\eqref{cad1}$ with \eqref{cad4}, and \eqref{cad2} with \eqref{cad3}, we obtain
\begin{equation}\label{A}
\mathcal{A}_1^1+\mathcal{A}_2^2=0,\qquad
\mathcal{A}_1^2-\mathcal{A}_2^1=0.
\end{equation}
For $\alpha\geq3$,  
\[
\langle \nabla^\perp_{e_i}{\bf H},\xi_\alpha\rangle =0.
\]
Hence
\begin{align*}
\langle \Delta^\perp{\bf H},\xi_\alpha\rangle 
&=-\sum_{i=1}^m\langle \nabla^\perp_{e_i}{\bf H}, \nabla^\perp_{e_i}\xi_\alpha \rangle \\
&=\mathcal{A}_1^1\mathcal{B}_1^\alpha
+\mathcal{A}_1^2\mathcal{C}_1^\alpha
+\mathcal{A}_2^1\mathcal{B}_2^\alpha
+\mathcal{A}_2^2\mathcal{C}_2^\alpha\\
&=\mathcal{B}_1^\alpha(\mathcal{A}_1^1+\mathcal{A}_2^2)
+\mathcal{B}_2^\alpha(-\mathcal{A}_1^2+\mathcal{A}_2^1)\\
&=0,
\end{align*}
where the last equality follows from \eqref{A}. 
Combining this with \eqref{normal2}, we obtain
\begin{align*}
0
&=\langle\Delta^\perp {\bf H}
-
m|{\bf H}|^2{\bf H}
-
2\mu^2{\bf H}_0
+
m\rho {\bf H}
,\xi_\alpha
 \rangle \\
&=
m(\rho-|{\bf H}|^2)\lambda_\alpha.
\end{align*}
Therefore, on $\tilde W$, $|{\bf H}|^2=\rho.$
By using \eqref{normal2} and 
\[
\langle \Delta^\perp{\bf H},{\bf H}\rangle
=-|\nabla^\perp {\bf H}|^2,
\]
we obtain
\begin{align*}
0
&=\langle\Delta^\perp {\bf H}
-
m|{\bf H}|^2{\bf H}
-
2\mu^2{\bf H}_0
+
m\rho {\bf H}
,{\bf H}
 \rangle \\
&=
-|\nabla^\perp{\bf H}|^2-2\mu^2\langle {\bf H}_0,{\bf H}\rangle \\
&=
-|\nabla^\perp{\bf H}|^2-2\mu^2|{\bf H}_0|^2.
\end{align*}
Therefore, 
\[
\nabla^\perp{\bf H}=0,\qquad 
{\bf H}_0=0.
\]
By the definition of $\tilde W$, ${\bf H}={\bf H}_0$ on $W\setminus \tilde W$.

\end{proof}

\begin{remark}
\begin{enumerate}
\item
In the first draft, the proof of Lemma \ref{bihwintchen} relied on \cite{DPSV2010}. The proof given in the present version does not rely on that reference.

\item
For $m\geq3$ and $p\geq2$, or $m=p=2$, Lemma \ref{bihwintchen} also implies that a biharmonic Wintgen ideal submanifold $M^m$ is a Chen submanifold. 

At a point where ${\bf H}\neq0$, choose an orthonormal normal frame $\{\eta_1=\frac{{\bf H}}{|{\bf H}|},\eta_2,\cdots,\eta_p\}$. 
By definition, 
\begin{align*}
\langle 
\sum_{i=1}^m B(A_{\bf H}e_i,e_i), \eta_\beta
\rangle 
&=
\sum_{i=1}^m\langle A_{\eta_\beta}A_{\bf H}e_i,e_i\rangle \\
&=
|{\bf H}|{\rm trace}(A_{\eta_1}A_{\eta_\beta}).
\end{align*}
Hence
\begin{align*}
a({\bf H})
&=\frac{|{\bf H}|}{m}\sum_{\beta=2}^p{\rm trace}(A_{\eta_1}A_{\eta_\beta})\eta_\beta\\
&=\frac{1}{m}\sum_{\beta=2}^p\langle \sum_{i=1}^m B(A_{\bf H}e_i,e_i), \eta_\beta \rangle\eta_\beta\\
&=\frac{1}{m}\sum_{\beta=2}^p\langle m|{\bf H}|^2{\bf H}+2\mu^2{\bf H}_0, \eta_\beta \rangle\eta_\beta\\
&=0,
\end{align*}
where the third equality follows from Lemma \ref{standard} and the last equality follows from Lemma \ref{bihwintchen} and the definition of $\{\eta_1,\cdots,\eta_p\}$.
At an umbilical point, where $\mu=0$, the same conclusion follows from Lemma \ref{standard}.
However, the biharmonicity assumption cannot be omitted in general. 
A Wintgen ideal submanifold need not be a Chen submanifold.
\end{enumerate}
\end{remark}

We provide an example of a Wintgen ideal submanifold in $\mathbb{E}^{m+3}$ that is not a Chen submanifold.
\begin{example}
Let $m\geq2$ and $U=(-1,1)^m$. Define $F:U\to\mathbb{E}^{m+3}$ by
\[
F(u,v,t_3,\cdots,t_m)=
\left(
u,v,t_3,\cdots,t_m,1+\frac{u^2-v^2}{2},uv,1
\right).
\]
Then the submanifold
\[
f(u,v,t_3,\cdots,t_m)
=\frac{2}{|F|^2}F
\]
with the induced metric is a Wintgen ideal submanifold but not a Chen submanifold.

Let $\{\bar e_1,\cdots,\bar e_{m+3}\}$ be the standard orthonormal frame of $\mathbb{E}^{m+3}$. Since 
\[
F=u\bar e_1+v\bar e_2+\sum_{i=3}^mt_i \bar e_i+\Big(1+\frac{u^2-v^2}{2}\Big)\bar e_{m+1}+uv\bar e_{m+2}+\bar e_{m+3},
\]
we have
\begin{align*}
F_u
&=\bar e_1+u\bar e_{m+1}+v\bar e_{m+2},\\
F_v
&=\bar e_2-v\bar e_{m+1}+u\bar e_{m+2},\\
F_{t_j}
&=\bar e_j.
\end{align*}
Then the induced metric of $F$ is 
\[
g_F=L(du^2+dv^2)+\sum_{j=3}^mdt_j^2,
\]
where $L=1+u^2+v^2$.
Set
\begin{align*}
E_1&=\frac{1}{\sqrt{L}}(\bar e_1+u\bar e_{m+1}+v\bar e_{m+2}),\\
E_2&=\frac{1}{\sqrt{L}}(\bar e_2-v\bar e_{m+1}+u\bar e_{m+2}),\\
E_j&=\bar e_j\qquad(3\leq j\leq m),\\
\nu_1&=\frac{1}{\sqrt{L}}(-v\bar e_1-u\bar e_2+\bar e_{m+2}),\\
\nu_2&=\frac{1}{\sqrt{L}}(-u\bar e_1+v\bar e_2+\bar e_{m+1}),\\
\nu_3&=\bar e_{m+3}.
\end{align*}
Then, $\{E_1,\cdots,E_m\}$ and $\{\nu_1,\nu_2,\nu_3\}$ are local orthonormal  tangent and normal frames along $F$, respectively.
A direct calculation shows that
\begin{align*}
A_{\nu_1}^F=\frac{1}{L^\frac{3}{2}}S_1,\qquad
A_{\nu_2}^F=\frac{1}{L^\frac{3}{2}}S_2,\qquad
A_{\nu_3}^F=0.
\end{align*}
Define a map $P$ as follows:
\[
P(Z)=Z-\frac{2}{D}\langle Z,F \rangle F,
\]
where $D=|F|^2.$
Then $df=\frac{2}{D}P(dF).$
Hence
\[g_f=\frac{4}{D^2}g_F.
\]
Let 
\[
e_i=P(E_i),\qquad
\xi_\alpha=P(\nu_\alpha).
\]
Then the shape operator can be written as 
\[
A^f_{\xi_\alpha}
=\frac{D}{2}A_{\nu_\alpha}^F
+\langle F,\nu_\alpha\rangle I_m.
\]
Set 
\[
\mu=\frac{D}{2}L^{-\frac{3}{2}},\qquad
\lambda_1=-L^{-\frac{1}{2}}uv,\qquad
\lambda_2=L^{-\frac{1}{2}}\Big(1-\frac{u^2-v^2}{2}\Big).
\]
Then 
\[
A_{\xi_1}^f=\lambda_1I_m+\mu S_1,\qquad
A_{\xi_2}^f=\lambda_2I_m+\mu S_2,\qquad
A_{\xi_3}^f=I_m.
\]
By the above notation, we obtain
\begin{align*}
|{\bf H}|^2
&=\lambda_1^2+\lambda_2^2+1,\\
|B|^2
&=m|{\bf H}|^2+4\mu^2,\\
R
&=\frac{m^2|{\bf H}|^2-|B|^2}{m(m-1)},\\
\nu^\perp
&=\frac{4\mu^2}{m(m-1)}.
\end{align*}
Hence
\[
R+\nu^\perp
=|{\bf H}|^2.
\]
Therefore the submanifold is a Wintgen ideal submanifold.
However, at the origin $o$, 
\[
D=2,\quad
L=1,\quad
\mu=1,\quad
\lambda_1=0,\quad
\lambda_2=1,
\]
\[
\xi_1=\bar e_{m+2},\quad
\xi_2=-\bar e_{m+3},\quad
\xi_3=-\bar e_{m+1},
\]
\[
A_{\xi_1}^f=S_1,\quad
A_{\xi_2}^f=S_2+I_m,\quad
A_{\xi_3}^f=I_m.
\]
Hence,
\[{\bf H}(o)=\xi_2+\xi_3.
\]
We also have
\[
\eta_1=\frac{1}{\sqrt{2}}(\xi_2+\xi_3),\quad
\eta_2=\frac{1}{\sqrt{2}}(\xi_2-\xi_3),\quad
\eta_3=\xi_1,
\]
hence
\[
A_{\eta_1}^f=\frac{1}{\sqrt{2}}(S_2+2I_m),\quad
A_{\eta_2}^f=\frac{1}{\sqrt{2}}S_2,\quad
A_{\eta_3}^f=S_1.
\]
Thus
\[
a({\bf H})(o)
=\frac{1}{m}(\xi_2-\xi_3)\neq0.
\] 
Therefore it is not a Chen submanifold. 
\end{example}

\section{Proof of the main theorem}
We prove Theorem \ref{main}.
When $p=1$, we are in the totally umbilical case, which is covered by the first part of the proof.
Thus, except in the totally umbilical case, we assume $p\geq2$ and use the Choi--Lu normal form.
As explained in the preliminaries, this normal form is used locally on open sets on which the adapted frame can be chosen smoothly (cf. \cite{CL2008}, \cite{DPSV2012}, \cite{GT2008}).

\begin{proof}[Proof of Theorem \ref{main}]

Let
\[
U=\{x\in M:{\bf H}(x)\neq0\}.
\]

\subsection*{The totally umbilical case}
Suppose first that there exists a nonempty connected open set $W\subset U$ on which the restriction of $M^m$ is totally umbilical.
Equivalently, in the Choi--Lu notation when $p\geq2$, this corresponds to
\[
\mu=0 \quad \text{on } W.
\]
Then
\[
B(X,Y)=\langle X,Y\rangle {\bf H}
\]
on $W$.
The Codazzi equation gives
\[
\langle Y,Z\rangle\nabla_X^\perp {\bf H}
=
\langle X,Z\rangle\nabla_Y^\perp {\bf H}.
\]
Choosing a tangent vector $Y$ with $X\perp Y$ and putting $Z=Y$, we obtain
\[
\nabla^\perp {\bf H}=0.
\]
Therefore $\Delta^\perp {\bf H}=0$.
The normal biharmonic equation gives
\[
-|{\bf H}|^2{\bf H}+\rho {\bf H}=0.
\]
Since $W\subset U$, we obtain
\begin{equation}\label{umbcase}
|{\bf H}|^2=\rho
\end{equation}
on $W$.
Thus this case is impossible when $\rho\leq0$, and it gives constant mean curvature when $\rho>0$.

\subsection*{The case where ${\bf H}={\bf H}_0$ on an open set}
Let $W$ be a nonempty connected open subset of $U$.
Assume that
\[
\mu\neq0 \quad \text{on } W
\qquad\text{and}\qquad
{\bf H}={\bf H}_0 \quad \text{on } W.
\]
The argument below is local on $W$.
Let $x\in W$ be arbitrary.
After replacing a neighbourhood of $x$ by a sufficiently small connected open neighbourhood, which we still denote by $W$, we may use a smooth adapted Choi--Lu frame on $W$.
By rotating the normal plane spanned by $\xi_1$ and $\xi_2$ and applying the corresponding rotation in the tangent plane spanned by $e_1$ and $e_2$, we may assume that
\[
{\bf H}=a\xi_1,
\qquad
a>0.
\]
Thus
\begin{equation}\label{lastshape}
A_{\xi_1}=aI_m+\mu S_1,
\qquad
A_{\xi_2}=\mu S_2,
\qquad
A_{\xi_\alpha}=0
\quad (\alpha\geq3).
\end{equation}
Write
\begin{equation}\label{normalconn1}
\nabla^\perp_{e_i}\xi_1
=
q_i\xi_2+
\sum_{\alpha=3}^pr_i^\alpha\xi_\alpha,
\end{equation}
and
\begin{equation}\label{normalconn2}
\nabla^\perp_{e_i}\xi_2
=
-q_i\xi_1+
\sum_{\alpha=3}^ps_i^\alpha\xi_\alpha.
\end{equation}
Put
\[
\theta_i=\langle \nabla_{e_i}e_1,e_2\rangle
\qquad (i=1,\cdots,m).
\]
The Codazzi equations for $B(e_1,e_1)$, $B(e_1,e_2)$, and $B(e_2,e_2)$ in the plane spanned by $e_1$ and $e_2$ give the following four relations.
\begin{equation}\label{plane1}
\mu_2=a_1+\mu q_1+2\mu\theta_1,
\end{equation}
\begin{equation}\label{plane2}
a q_1-\mu_1=\mu q_2+2\mu\theta_2,
\end{equation}
\begin{equation}\label{plane3}
\mu_1=a_2-\mu q_2-2\mu\theta_2,
\end{equation}
and
\begin{equation}\label{plane4}
\mu q_1+2\mu\theta_1=a q_2+\mu_2,
\end{equation}
where $a_i=e_i(a)$ and $\mu_i=e_i(\mu)$.
Indeed, these are, respectively, the $\xi_1$- and $\xi_2$-components of
\[
(\nabla_{e_1}B)(e_2,e_2)=(\nabla_{e_2}B)(e_1,e_2)
\]
and
\[
(\nabla_{e_1}B)(e_1,e_2)=(\nabla_{e_2}B)(e_1,e_1).
\]
Adding \eqref{plane1} and \eqref{plane4}, and adding \eqref{plane2} and \eqref{plane3}, we obtain
\begin{equation}\label{q12}
q_2=-\frac{a_1}{a},
\qquad
q_1=\frac{a_2}{a}.
\end{equation}

Since ${\bf H}=a\xi_1$, we have
\[
\nabla^\perp_{e_i}{\bf H}
=
a_i\xi_1+aq_i\xi_2+
 a\sum_{\alpha=3}^pr_i^\alpha\xi_\alpha.
\]
Since $A_{\xi_\alpha}=0$ for $\alpha\geq3$, the tangential equation \eqref{tangent} gives
\begin{equation}\label{tan1}
(m+2)aa_1+2\mu a_2+2a\mu q_1=0,
\end{equation}
\begin{equation}\label{tan2}
2\mu a_1+(m+2)aa_2-2a\mu q_2=0
\end{equation}
and
\begin{equation}\label{tanu}
a_u=0
\quad (u=3,\cdots,m),
\end{equation}
where \eqref{tanu} is an empty condition when $m=2$.
Substituting \eqref{q12} into \eqref{tan1} and \eqref{tan2}, we obtain
\begin{equation}\label{tan3}
(m+2)aa_1+4\mu a_2=0,
\qquad
4\mu a_1+(m+2)aa_2=0.
\end{equation}

Let
\[
D=\{x\in W:(m+2)^2a^2(x)\neq16\mu^2(x)\}.
\]
This is an open subset of $W$.
If it is nonempty, then on each nonempty connected component of $D$, equation \eqref{tan3} gives $a_1=a_2=0$.
Together with \eqref{tanu}, this gives $\nabla a=0$ on that component.
Thus $|{\bf H}|=a$ is constant on each connected component of $D$.
Taking the inner product of \eqref{normal2} with ${\bf H}=a\xi_1$ gives
\begin{equation}\label{nondegeq}
|\nabla^\perp {\bf H}|^2
+
(ma^2+2\mu^2-m\rho)a^2=0
\end{equation}
on each connected component of $D$.
This is impossible when $\rho\leq0$.
When $\rho>0$, it gives local constancy of $|{\bf H}|$ on $D$.

It remains to treat the local alternative in which the degenerate relation
\begin{equation}\label{degenerate}
(m+2)^2a^2=16\mu^2
\end{equation}
holds.
More precisely, let $x$ be an arbitrary point of the interior, relative to $W$, of the set on which \eqref{degenerate} holds.
After replacing a neighbourhood of $x$ by a sufficiently small connected open neighbourhood, we obtain a nonempty connected open set $W_0$ on which \eqref{degenerate} holds identically.
Put
\[
k=\frac{m+2}{4}.
\]
For some $\varepsilon=\pm1$, we have
\begin{equation}\label{mueps}
\mu=\varepsilon ka
\end{equation}
on $W_0$.
Since $x$ was arbitrary, the following local calculation applies near every point of the interior of the degenerate set.
In the following calculation, all equations are understood on $W_0$.
Set $q=q_2$. By \eqref{q12} and \eqref{tan3},
\begin{equation}\label{degq}
a_1=-aq,
\qquad
a_2=\varepsilon aq,
\qquad
q_1=\varepsilon q.
\end{equation}
Equations \eqref{plane1} and \eqref{plane3}, together with \eqref{mueps} and \eqref{degq}, give
\begin{equation}\label{theta12}
\theta_1=\frac{\varepsilon q}{2k},
\qquad
\theta_2=\frac{q}{2k}.
\end{equation}

We now split the analysis of this degenerate case according to whether $m\geq3$ or $m=2$.

Suppose first that $m\geq3$. Let $u\geq3$.
We compute the Codazzi equations involving one vector $e_u$.
The $\xi_2$-component of
\[
(\nabla_{e_u}B)(e_1,e_1)=(\nabla_{e_1}B)(e_u,e_1)
\]
gives
\begin{equation}\label{u11}
a q_u+\mu_u=\mu\omega_1^u(e_1),
\end{equation}
where $\omega_i^j(X)=\langle\nabla_Xe_i,e_j\rangle$.
The $\xi_1$-component of
\[
(\nabla_{e_u}B)(e_1,e_2)=(\nabla_{e_1}B)(e_u,e_2)
\]
gives
\begin{equation}\label{u12}
\mu_u=\mu\omega_1^u(e_1).
\end{equation}
Since \eqref{tanu} and \eqref{mueps} imply $\mu_u=0$, equations \eqref{u11} and \eqref{u12} give
\begin{equation}\label{quzero}
q_u=0,
\qquad
\omega_1^u(e_1)=0.
\end{equation}
The $\xi_2$-component of
\[
(\nabla_{e_u}B)(e_1,e_2)=(\nabla_{e_1}B)(e_u,e_2)
\]
then gives
\begin{equation}\label{omega2u1}
\omega_2^u(e_1)=-2\theta_u.
\end{equation}
Similarly, the same argument applied to the pairs involving $e_2$ gives
\begin{equation}\label{omega1u2}
\omega_1^u(e_2)=2\theta_u,
\qquad
\omega_2^u(e_2)=0.
\end{equation}
Consequently,
\begin{equation}\label{bracket}
[e_1,e_2]
=
-\theta_1e_1
-
\theta_2e_2
-
4\sum_{u=3}^m\theta_u e_u.
\end{equation}

For each $\alpha\geq3$, the Codazzi equation for $A_{\xi_\alpha}=0$ gives
\begin{equation}\label{rszero}
r_i^\alpha=s_i^\alpha=0
\qquad
(i=1,\cdots,m).
\end{equation}
Indeed, the shape-operator form of the Codazzi equation gives
\[
0
=
-r_i^\alpha A_{\xi_1}e_j-s_i^\alpha A_{\xi_2}e_j
+r_j^\alpha A_{\xi_1}e_i+s_j^\alpha A_{\xi_2}e_i.
\]
Taking $(i,j)=(u,1)$ yields
\[
r_1^\alpha=0,
\qquad
r_u^\alpha=0,
\qquad
s_u^\alpha=0.
\]
Taking $(i,j)=(u,2)$ yields $r_2^\alpha=0$.
Finally, taking $(i,j)=(1,2)$ gives $s_1^\alpha=s_2^\alpha=0$.
Thus \eqref{rszero} follows.

Let $q^\perp$ be the normal connection form defined by
\[
q^\perp(X)=\langle\nabla_X^\perp\xi_1,\xi_2\rangle.
\]
Let $r^\alpha$ and $s^\alpha$ be the one-forms defined by
\[
r^\alpha(e_i)=r_i^\alpha,
\qquad
s^\alpha(e_i)=s_i^\alpha,
\]
and let $\{\omega^1,\cdots,\omega^m\}$ be the dual coframe of $\{e_1,\cdots,e_m\}$.
By \eqref{degq} and \eqref{quzero},
\begin{equation}\label{p15}
q^\perp(e_1)=\varepsilon q,
\qquad
q^\perp(e_2)=q,
\qquad
q^\perp(e_u)=0.
\end{equation}
Set $\theta(X)=\langle\nabla_Xe_1,e_2\rangle$. Using \eqref{theta12} and \eqref{bracket}, we obtain
\begin{equation}\label{dqtheta}
dq^\perp(e_1,e_2)
=
2k\left\{
d\theta(e_1,e_2)-4\sum_{u=3}^m\theta_u^2
\right\}.
\end{equation}
Indeed, $q^\perp$ and $2k\theta$ have the same values on $e_1$ and $e_2$, but need not agree on $e_u$.
The correction term in \eqref{dqtheta} comes from the $e_u$-component of the bracket \eqref{bracket}.

The Gauss equation for the plane spanned by $e_1$ and $e_2$ gives
\[
d\theta(e_1,e_2)
-
\sum_{u=3}^m\omega_1^u\wedge\omega_u^2(e_1,e_2)
=
2\mu^2-a^2-\rho.
\]
By \eqref{quzero}, \eqref{omega2u1} and \eqref{omega1u2},
\[
\omega_1^u\wedge\omega_u^2(e_1,e_2)
=
-4\theta_u^2.
\]
Therefore
\begin{equation}\label{gaussdeg}
d\theta(e_1,e_2)+4\sum_{u=3}^m\theta_u^2
=
2\mu^2-a^2-\rho.
\end{equation}
Combining \eqref{dqtheta} and \eqref{gaussdeg}, we obtain
\begin{equation}\label{dq1}
dq^\perp(e_1,e_2)
=
2k\left(2\mu^2-a^2-\rho-8\sum_{u=3}^m\theta_u^2\right).
\end{equation}
On the other hand, the Ricci equation for the normal pair $(\xi_1,\xi_2)$ gives
\[
dq^\perp
+
\sum_{\alpha=3}^pr^\alpha\wedge s^\alpha
=
2\mu^2\omega^1\wedge\omega^2.
\]
By \eqref{rszero}, this reduces to
\begin{equation}\label{dq2}
dq^\perp(e_1,e_2)=2\mu^2.
\end{equation}

We will show $\theta_u=0$. 
We obtain
\begin{align*}
0
&=[e_u,e_1+\varepsilon e_2]a\\
&=(\nabla_{e_u}(e_1+\varepsilon e_2)-\nabla_{e_1+\varepsilon e_2}e_u)a\\
&=-2\varepsilon aq\theta_u,
\end{align*}
where the last equality follows from 
\eqref{tanu}, \eqref{degq}, \eqref{quzero},
\eqref{omega2u1}, and \eqref{omega1u2}.
Hence $q\theta_u=0$. Assume that $\theta_u\neq0$ on some small open set $\Omega_1$. Then $q=0$ on $\Omega_1$.  
By \eqref{p15}, we have $q^\perp=0$. However, \eqref{dq2} yields a contradiction. Hence $\theta_u=0.$ 
Substituting this into \eqref{dq1}, we have
\begin{equation}\label{dq3}
dq^\perp(e_1,e_2)
=
2k
\left(
2\mu^2-a^2-\rho
\right).
\end{equation}

Combining \eqref{dq3} and \eqref{dq2}, and using $\mu^2=k^2a^2$, we obtain
\begin{equation}\label{rhoeq}
\rho=(2k+1)(k-1)a^2.
\end{equation}
Since $m\geq3$, we have $k>1$.
If $\rho\leq0$, equation \eqref{rhoeq} is impossible.
If $\rho>0$, equation \eqref{rhoeq} gives
\[
a^2=\frac{\rho}{(2k+1)(k-1)}.
\]
Thus $a=|{\bf H}|$ is constant on $W_0$.

Suppose next that $m=p=2$.
Then $k=1$, and the preceding $m\geq3$ calculation cannot be used.
Indeed, the equations involving $e_u$ do not exist, and the coefficient $(2k+1)(k-1)$ in \eqref{rhoeq} becomes zero.

In this case \eqref{degenerate} is $a^2=\mu^2$, and \eqref{mueps} is $\mu=\varepsilon a$.
As before, let 
\[
q^\perp(X)=\langle\nabla_X^\perp\xi_1,\xi_2\rangle,\qquad \theta(X)=\langle\nabla_Xe_1,e_2\rangle.
\]
By \eqref{degq} and \eqref{theta12} with $k=1$, we have
\[
q^\perp=2\theta.
\]
Since there is no tangent direction $e_u$ for $u\geq3$, the Gauss equation for the plane spanned by $e_1$ and $e_2$ is simply
\[
d\theta(e_1,e_2)=2\mu^2-a^2-\rho.
\]
Using $\mu^2=a^2$, we obtain
\[
d\theta(e_1,e_2)=a^2-\rho,
\]
and therefore
\begin{equation}\label{surfdq1}
dq^\perp(e_1,e_2)=2a^2-2\rho.
\end{equation}
On the other hand, the Ricci equation for the normal pair $(\xi_1,\xi_2)$ gives
\begin{equation}\label{dqp}
dq^\perp(e_1,e_2)
=
2\mu^2.
\end{equation}
Combining $a^2=\mu^2$, \eqref{surfdq1}, and \eqref{dqp}, we obtain
\[
\rho=0.
\] 

Since $\mu^2=a^2$ and ${\bf H}=a\xi_1$, the normal biharmonic equation becomes
\begin{equation}\label{surfnormal3}
\Delta^\perp {\bf H}=4a^3\xi_1.
\end{equation}
The $\xi_1$-component of $\Delta^\perp {\bf H}$ is
\[
a\left(
-e_1(q)+\varepsilon e_2(q)-q^2
\right).
\]
Thus \eqref{surfnormal3} gives
\begin{equation}\label{surfqeq1}
-e_1(q)+\varepsilon e_2(q)-q^2
=4a^2.
\end{equation}
On the other hand, using \eqref{theta12} with $k=1$ and
\[
[e_1,e_2]=-\theta_1e_1-\theta_2e_2,
\]
we obtain
\[
d\theta(e_1,e_2)
=
\frac12e_1(q)-\frac{\varepsilon}{2}e_2(q)+\frac12q^2.
\]
Since $d\theta(e_1,e_2)=a^2$, we obtain
\[
-e_1(q)+\varepsilon e_2(q)-q^2=-2a^2.
\]
Comparing this equality with \eqref{surfqeq1}, we obtain
\begin{equation}\label{surfrhoeq2}
a^2
=0,
\end{equation}
which is a contradiction.

We record the consequences of these two subcases for the present local case where ${\bf H}={\bf H}_0$.
If $\rho\leq0$, then this case cannot occur near any point of $W$.
Indeed, if $D$ is nonempty, then \eqref{nondegeq} gives a contradiction on each connected component of $D$.
Hence $D$ is empty, and therefore \eqref{degenerate} holds on $W$.
If $m\geq3$, the preceding calculation in the degenerate case gives the contradiction \eqref{rhoeq}.
If $m=2$, it gives the contradiction \eqref{surfrhoeq2}.

If $\rho>0$, the case $m=2$ cannot occur on $D$. In fact, since $a$ is constant, by \eqref{q12}, $q^\perp=0$. However, the Ricci equation provides a contradiction. 
Hence, $D$ is empty and the above argument shows that we have a contradiction.

We only need to consider $m\geq3$.
We can show $a=|{\bf H}|$ is locally constant on $W$.
Indeed, let
\[
G_W=\{x\in W:\nabla a(x)\neq0\}.
\]
This is an open subset of $W$.
Suppose that $G_W$ is nonempty, and let $x\in G_W$ be arbitrary.
After replacing a neighbourhood of $x$ by a sufficiently small connected open neighbourhood, we obtain a nonempty connected open set $W_1\subset G_W$.
The set $W_1\cap D$ must be empty, because $\nabla a=0$ on every connected component of $D$.
Hence \eqref{degenerate} holds at every point of $W_1$.
Since $W_1$ is open, the relation \eqref{degenerate} holds identically on $W_1$.
The calculation in the degenerate case forces $a$ to be constant on $W_1$, which contradicts $W_1\subset G_W$.
Since $x\in G_W$ was arbitrary, $G_W$ is empty.

\subsection*{Conclusion when $\rho\leq0$}
Assume that $\rho\leq0$.
We prove that $U$ is empty.
Suppose, to the contrary, that $U$ is nonempty.

The totally umbilical case was already shown to be impossible on any nonempty open subset of $U$.
Hence it remains to rule out non-umbilical points of $U$.
Suppose first that there exists a non-umbilical point $x\in U$.
Then, after replacing a neighbourhood of $x$ by a sufficiently small connected open neighbourhood $V\subset U$, we may use a smooth adapted Choi--Lu frame on $V$ with
\[
\mu\neq0 \quad \text{on } V.
\]
Set
\[
V_1=\{x\in V:{\bf H}_0(x)\neq0\}.
\]
If $V_1$ is nonempty, then on each connected component $W'$ of $V_1$, Lemma \ref{bihwintchen} gives ${\bf H}={\bf H}_0$ on $W'$.
The case where ${\bf H}={\bf H}_0$ on an open set gives a contradiction.
Therefore $V_1$ is empty, and hence ${\bf H}_0=0$ on $V$.
Lemma \ref{bihwintchen} gives a contradiction.
This contradiction shows that no non-umbilical point exists in $U$.

Consequently, every point of $U$ is umbilical.
Since $U$ is open, the restriction of $M^m$ to $U$ is totally umbilical, which contradicts the totally umbilical case.
Therefore $U$ is empty.
Hence ${\bf H}=0$ on $M^m$, and $M^m$ is minimal.

\subsection*{Conclusion when $\rho>0$}
Assume that $\rho>0$.
We prove that $|{\bf H}|$ is constant on each connected component of $M^m$.
It is enough first to prove that $|{\bf H}|$ is locally constant on $U$.
Let
\[
G=\{x\in U:\nabla |{\bf H}|^2(x)\neq0\}.
\]
This is an open subset of $U$.
Suppose that $G$ is nonempty.

If the restriction of $M^m$ to some nonempty open subset of $G$ is totally umbilical, then \eqref{umbcase} gives $|{\bf H}|^2=\rho$ on that open subset, contradicting its inclusion in $G$.
Hence it remains to rule out non-umbilical points of $G$.
Suppose first that there exists a non-umbilical point $x\in G$.
After replacing a neighbourhood of $x$ by a sufficiently small connected open neighbourhood $V\subset G$, we may use a smooth adapted Choi--Lu frame on $V$ with
\[
\mu\neq0 \quad \text{on } V.
\]
Let
\[
V_1=\{x\in V:{\bf H}_0(x)\neq0\}.
\]
If $V_1$ is nonempty, then on each connected component $W'$ of $V_1$, Lemma \ref{bihwintchen} gives ${\bf H}={\bf H}_0$ on $W'$.
The result proved above for the case where ${\bf H}={\bf H}_0$ on an open set shows that $|{\bf H}|$ is locally constant on $W'$, contradicting $W'\subset G$.
Therefore $V_1$ is empty, and hence ${\bf H}_0=0$ on $V$.
Lemma \ref{bihwintchen} shows that $|{\bf H}|$ is constant on $V$, again contradicting $V\subset G$.
This contradiction shows that no non-umbilical point exists in $G$.

Therefore every point of $G$ is umbilical.
Since $G$ is open, the restriction of $M^m$ to $G$ is totally umbilical.
By \eqref{umbcase}, $|{\bf H}|^2=\rho$ on $G$, contradicting the definition of $G$.
Hence $G$ is empty.
Therefore
\[
\nabla |{\bf H}|^2=0
\]
on $U$.

Let $C$ be a connected component of $M^m$.
If $C\cap U$ is empty, then ${\bf H}=0$ on $C$.
Assume that $C\cap U$ is nonempty, and let $\Omega$ be a connected component of $C\cap U$.
Since $\nabla |{\bf H}|^2=0$ on $U$, the function $|{\bf H}|$ is constant on $\Omega$.
Let this constant be $h_0$.
Since $\Omega\subset U$, we have $h_0>0$.
If the relative boundary $\partial_C\Omega$ in $C$ were nonempty, then
\[
\partial_C\Omega\subset C\setminus U.
\]
Thus ${\bf H}=0$ on $\partial_C\Omega$.
Taking a sequence in $\Omega$ converging to a boundary point and using the continuity of $|{\bf H}|$, we obtain $h_0=0$.
This contradicts $h_0>0$.
Hence $\partial_C\Omega$ is empty.
Since $C$ is connected, $\Omega=C$.
Therefore $|{\bf H}|$ is constant on $C$.
This proves Theorem \ref{main}.
\end{proof}

\begin{remark}
When $\rho=0$, Theorem \ref{main} gives the Euclidean conclusion in the class of Wintgen ideal submanifolds under the dimensional assumptions stated there.
When $\rho<0$, it gives the corresponding hyperbolic conclusion.
The conclusion cannot be improved to minimality in the spherical case, since small hyperspheres $\mathbb{S}^{n-1}\left(\frac{1}{\sqrt{2}}\right)\subset \mathbb{S}^{n}$ in spheres give proper biharmonic totally umbilical examples \cite{Jiang1986a}.
\end{remark}

\subsection*{Acknowledgement}
The author would like to thank Zhenxiao Xie for bringing his classification result on Wintgen ideal submanifolds with constant mean curvature to the author's attention and for his helpful comments concerning the main theorem of this paper.

\end{document}